\documentclass{amsart}

\usepackage[ascii]{inputenc}
\usepackage[ngerman,USenglish]{babel}
\usepackage{microtype}
\usepackage{amsmath, amssymb}
\usepackage{amsfonts}
\usepackage{amsthm}
\usepackage{aliascnt}
\usepackage{tikz}
\usepackage[bookmarksnumbered]{hyperref}
\usepackage[all]{hypcap}

\title[Generating Functions and Symmetric Cones]{Lattice~Point Generating~Functions and Symmetric Cones}

\author{Matthias Beck}
\address{Department of Mathematics\\
         San Francisco State University\\
         San Francisco, CA 94132}
\email{mattbeck@sfsu.edu}

\author{Thomas Bliem}
\address{Kempener Str.\ 57\\
         50733 K\"oln\\
         Germany}
\email{tbliem@gmx.de}

\author{Benjamin Braun}
\address{Department of Mathematics\\
         University of Kentucky\\
         Lexington, KY 40506--0027}
\email{benjamin.braun@uky.edu}

\author{Carla Savage}
\address{Department of Computer Science\\
         North Carolina State University \\
         Raleigh, NC 27695-8206}
\email{savage@ncsu.edu}

\date{}

\allowdisplaybreaks[4]

\newcommand{\newjointcountertheorem}[3]{
	\newaliascnt{#1}{#2}
	\newtheorem{#1}[#1]{#3}
	\aliascntresetthe{#1}	
}

\newtheorem{theorem}{Theorem}[section]
\newjointcountertheorem{proposition}{theorem}{Proposition}
\newjointcountertheorem{lemma}{theorem}{Lemma}
\newjointcountertheorem{corollary}{theorem}{Corollary}
\theoremstyle{definition}
\newjointcountertheorem{definition}{theorem}{Definition}
\newjointcountertheorem{example}{theorem}{Example}
\newjointcountertheorem{remark}{theorem}{Remark}

\newcommand{\des}{\mathrm{des}} 
\newcommand{\maj}{\mathrm{maj}} 
\newcommand{\comaj}{\mathrm{comaj}} 
\renewcommand{\phi}{\varphi} 
\renewcommand{\epsilon}{\varepsilon} 
\newcommand{\definedterm}{\textbf} 
\newcommand{\defeq}{:=} 
\newcommand{\Ehr}{\mathrm{Ehr}} 
\newcommand{\Dr}{D_\mathrm{r}} 
\DeclareMathOperator{\stat}{stat}
\DeclareMathOperator{\cobin}{cobin}

\begin{document}

\begin{abstract}
	We show that a recent identity of Beck--Gessel--Lee--Savage on the generating function of symmetrically contrained compositions of integers generalizes naturally to a family of convex polyhedral cones that are invariant under the action of a finite reflection group.
	We obtain general expressions for the multivariate generating functions of such cones, and work out the specific cases of a symmetry group of type A (previously known) and types B and D (new).
	We obtain several applications of the special cases in type B, including identities involving permutation statistics and lecture hall partitions.
\end{abstract}

\maketitle

\section{Introduction}

	Motivated by the ``constrained compositions'' introduced by Andrews--Paule--Riese
	\cite{PAVII}, Beck--Gessel--Lee--Savage \cite{BeckGesselLeeSavage} enumerated
	\definedterm{symmetrically constrained compositions}, i.e., compositions
	of an integer $M$ into $n$ nonnegative parts
	\[
		M = \lambda_1 + \lambda_2 + \cdots + \lambda_n \, ,
	\]
	where the sequence $\lambda\defeq\left( \lambda_1, \lambda_2, \dots, \lambda_n
	\right)$ satisfies the symmetric system of linear inequalities
	\[
		a_1 \lambda_{ \pi(1) } + a_2 \lambda_{ \pi(2) } + \dots + a_n \lambda_{ \pi(n) } \ge 0
		\qquad \text{ for all } \pi \in S_n \, .
	\]
	Specifically, \cite{BeckGesselLeeSavage} discusses various approaches to compute, for a fixed
	set of parameters $a_1, a_2, \dots, a_n$, the generating functions
	\[
		F \left( z_1, z_2, \dots, z_n \right) \defeq \sum_\lambda z_1^{ \lambda_1 }
	z_2^{ \lambda_2 } \cdots z_n^{ \lambda_n }
	\]
	and
	\[
		F(q) \defeq F(q,q,\dots,q) = \sum_\lambda q^{ \lambda_1 + \lambda_2 + \dots + \lambda_n } ,
	\]
	where both sums extend over all symmetrically constrained compositions $\lambda$.
	One viewpoint of \cite{BeckGesselLeeSavage} is geometric: The
	compositions $\left( \lambda_1, \lambda_2, \dots, \lambda_n \right)$
	are interpreted as integer lattice points in the cone
	\begin{equation}\label{coneq}
		\left\{ x \in \mathbf{R}^n \mid \forall \, \sigma \in W : (\sigma x, a) \ge 0 \right\} ,
	\end{equation}
	where $W$ is the image of the permutation representation of $S_n$,
	$a = (a_1, \ldots, a_n)$,
	and $(\;\;{,}\;\;)$ is the standard inner product on $\mathbf{R}^n$.
	This viewpoint together with permutation statistics of $S_n$ gave rise
	to explicit (and in some instances surprising) generating function
	formulas.

	Our goal is to generalize the results in \cite{BeckGesselLeeSavage} to cones of
	the form \eqref{coneq} where $W$ is another reflection group.
	In addition to obtaining general multivariate generating function identities, we obtain several applications of these results for hyperoctahedral groups.
	These applications are similar in spirit to the applications in the symmetric-group case found in \cite{BeckGesselLeeSavage}.

	The outline of our paper is as follows.
	The general setup for our approach is discussed in the next section, which also contains our central result, \autoref{vbvXkhNC}. 
	Section~\ref{AKg4ecYc} illustrates our approach by re-deriving the main result in \cite{BeckGesselLeeSavage}.
	Sections~\ref{dxuwjrHR} and~\ref{XC4vSQKD} consider cones constrained by reflection groups of type $B$ and $D$, respectively.
	Further, Section~\ref{dxuwjrHR} contains applications obtained through specializing our generating functions in the type-$B$ case.

	\bigskip

	\emph{Acknowledgements}.
		M.B. is partially supported by the NSF (DMS-0810105).
		T.B. has been supported by the \foreignlanguage{ngerman}{Deutsche Forschungsgemeinschaft} (SPP 1388).
		B.B. is partially supported by the NSF (DMS-0758321).
		The authors thank the American Institute of Mathematics for support of our SQuaRE working group on ``Polyhedral Geometry and Partition Theory.'' 

	\bigskip
	
	\emph{2010 Mathematics Subject Classification}.
		05A15 Exact enumeration problems, generating functions;
		51F15 Reflection groups, reflection geometries;
		05A17 Partitions of integers;
		52B15 Symmetry properties of polytopes.

	\bigskip

	\emph{Keywords}.
		lattice point generating function,
		polyhedral cone,
		finite reflection group,
		Coxeter group,
		symmetrically constrained composition,
		permutation statistics,
		lecture hall partition.

\section{General theory} \label{sH4iHn8E}

	Our goal in this section is to study integer points in cones that are constrained by the orbit of a single linear constraint under an appropriate group action on real space.
	This goal is realized in \autoref{vbvXkhNC}, where the multivariate generating function encoding the integer points in such a cone is expressed as a sum of simpler generating functions.
	\autoref{vbvXkhNC} is an algebraic consequence of a geometric triangulation of the symmetric cone, which we obtain in \autoref{8UQHag7u}.
	\autoref{JNSFSTk5} makes the triangulation disjoint by using combinatorics of Coxeter groups as a tiebreaker for the walls separating the maximal cones in the triangulation.
	This is critical for our subsequent applications.

\subsection{Almost irreducible finite reflection groups, Coxeter groups, and descents}

	In the following, we will consider finite reflection groups (see, e.g., \cite{Humphreys} for background) that act on spaces in a restricted fashion.
	Namely, a finite reflection group $W \subset \mathrm{O}(V)$ acting on a Euclidean vector space $V$ is called \definedterm{almost irreducible} if $V$ decomposes into $W$-invariant subspaces $V = V_1 \oplus V_2$ such that $W$ acts irreducibly and nontrivially on $V_1$ and trivially on $V_2$, and that $V_2$ is $1$-dimensional.

	\begin{example}
		$S_n$ acts almost irreducibly on $\mathbf{R}^n$ by permutation of the components.
		The irreducible summand consists of all vectors with component sum $0$, and the trivial summand consists of all vectors with equal components.
		This is the case considered in \cite{BeckGesselLeeSavage}. 
	\end{example}

	\begin{example}
		Let $V_1$ be a Euclidean vector space and $W \subset \mathrm{O}(V_1)$ a nontrivial irreducible reflection group, i.e., a nontrivial reflection group such that $V_1$ does not contain any nontrivial proper $W$-invariant subspaces.
		Let $W$ act trivially on $\mathbf{R}$ and set $V = V_1 \oplus \mathbf{R}$.
		Then $W$ acts almost irreducibly on $V$.
	\end{example}

	A \definedterm{Coxeter group} of rank $r$ is a group admitting a presentation with generators $s_1, \ldots, s_r$ and relations $(s_js_k)^{m_{jk}} = 1$ for $m_{jk} \in \{1, 2, 3, \ldots \} \cup \{ \infty \}$ subject to the conditions that $m_{jk} = m_{kj}$ and $m_{jk} = 1 \iff j=k$.
	Here, a value of $m_{jk} = \infty$ is to be understood as the absence of the corresponding relation.
	Such generators are called \definedterm{simple generators}.
	For each Coxeter group considered, we will suppose that simple generators have been fixed once and for all.
	We refer the reader to \cite{BjornerBrenti} or \cite{Humphreys} for further information about Coxeter groups and their relation to reflection groups.

	The \definedterm{length} $l(\sigma)$ of an element $\sigma \in W$ of a Coxeter group $W$ is the smallest integer such that there is a decomposition $\sigma = s_{j_1} \cdots s_{j_{l(\sigma)}}$ of $\sigma$ as a product of $l(\sigma)$ not necessarily distinct simple generators.
	For any $\sigma \in W$, the \definedterm{right descent set} of $\sigma$ is
	\begin{equation} \label{iz7umNWf}
		\Dr(\sigma) \defeq \{ j \in \{ 1, \ldots, r \} \mid l(\sigma s_j) < l(\sigma) \} \, .
	\end{equation}

	\begin{remark}
	Propositions~\ref{prop:typeAdes}, \ref{prop:typeBdes}, and~\ref{prop:typeDdes} review the connection between the definition of descent given here and definitions of descent for Coxeter groups of types $A$, $B$, and $D$ in terms of the one-line notation.
	\end{remark}

	Recall that if $W$ is a finite reflection group, it is automatically a Coxeter group.
	Simple generators can be found as follows.
	Let $\mathcal{H}$ be the union of all reflection hyperplanes for $W$; denote by $F$ the closure of a connected component in $V\setminus \mathcal{H}$.
	It is immediate that $F$ is a convex polyhedral cone.
	Let $H_1, \ldots, H_r$ be the facet hyperplanes of $F$ and let $s_i$ be the reflection at $H_i$.
	Then $s_1, \ldots, s_r$ are simple generators of the Coxeter group $W$.
	See \cite[V.3.2, Th.~1]{Bourbaki} for the proof of these statements.

	A subset $F \subset V$ is a \definedterm{fundamental domain} for $W$ if $F$ is the closure of an open set and each $W$-orbit intersects $F$ in exactly one point.
	By \cite[Section I.12]{Humphreys}, every such $F$ is polyhedral, and is bounded by hyperplanes fixed by a set of simple reflections in $W$.
	Through the rest of this paper, when given a set of simple generators $s_1,\ldots,s_r$ of a reflection group $W$, we denote by $F$ a fixed fundamental domain with bounding hyperplanes corresponding to $s_1,\ldots,s_r$.
	
\subsection{Triangulations of monoconditional cones}

	Denote the value of a linear form $\phi \in V^*$ on a vector $x \in V$ by $\langle x, \phi \rangle$.
	Let $W \subset \mathrm{O}(V)$ be an almost irreducible reflection group.
	A \definedterm{symmetric cone} $C \subset V$ is a convex polyhedral cone that is $W$-invariant.
	A symmetric cone is called \definedterm{monoconditional} if there is a linear form $\phi \in V^*$, such that $V_1, V_2 \not\subset \ker(\phi)$ and
	\begin{equation} \label{BEyS3CWg}
		C = \{ x \in V \mid \forall \, \sigma \in W : \langle \sigma x, \phi \rangle \geq 0 \} \, .
	\end{equation}
	This generalizes \eqref{coneq}.
		
	\begin{example}
		The positive orthant $\mathbf{R}_{\geq 0}^n$ 
	is a monoconditional symmetric cone for the almost irreducible action of $S_n$ on $\mathbf{R}^n$ by permutation of the components.
		A possible linear form defining it is the projection on the first component.
	\end{example}

	\begin{lemma} \label{8UQHag7u}
		Let $W \subset \mathrm{O}(V)$ be an almost irreducible reflection group.
		Let $C \subset V$ be a monoconditional symmetric cone.	
		Let $F \subset V$ be a fundamental domain for the action of $W$ on $V$.
		Then the cone $C_+ \defeq C \cap F$ is simplicial.
		In particular, $C$ admits the triangulation
		\[
			C = \bigcup_{\sigma \in W} \sigma \, C_+ \, .
		\]
	\end{lemma}

	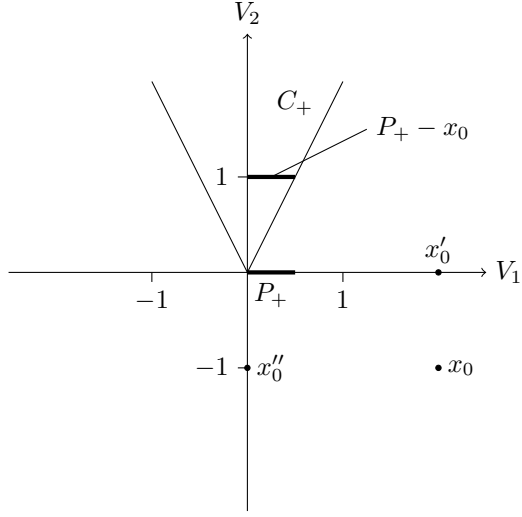
\begin{figure}
		\centering
		\begin{tikzpicture}[x=.5in,y=.5in]
			\draw[->] (-2.5,0) -- (2.5,0) node[right]{$V_1$};
			\draw[->] (0,-2.5) -- (0,2.5) node[above]{$V_2$};
			\draw (1,0) -- (1,-.1) node[below]{$1$};
			\draw (-1,0) -- (-1,-.1) node[below]{$-1$};
			\draw (0,1) -- (-.1,1) node[left]{$1$};
			\draw (0,-1) -- (-.1,-1) node[left]{$-1$};
			\draw (-1,2) -- (0,0) -- (1,2);
			\draw (.5,2) node[below]{$C_+$};
			\filldraw (2,-1) circle(1bp) node[right]{$x_0$};
			\filldraw (2,0) circle(1bp) node[above]{$x'_0$};
			\filldraw (0,-1) circle(1bp) node[right]{$x''_0$};
			\draw[ultra thick] (0,0) -- (.5,0) (.25,0) node[below]{$P_+$};
			\draw[ultra thick] (0,1) -- (.5,1);
			\draw (.25,1) -- (1.25,1.5) node[right]{$P_+ - x_0$};
		\end{tikzpicture}
		\caption{Notation used in the proof of \autoref{8UQHag7u}. The cyclic group of order $2$ acts almost irreducibly on $\mathbf{R}^2$ by sign change in the first component.}
		\label{5y7q4McK}
	\end{figure}

	\begin{proof}
		Some of the notation used in this proof is shown in \autoref{5y7q4McK} for convenience.
		Let $\phi \in V^*$ be a linear form defining $C$ as in \eqref{BEyS3CWg}.
		Let $x_\phi \in V$ such that $\langle x, \phi \rangle = -(x, x_\phi)$ for all $x \in V$.
		Let $x_0$ be the unique element of $Wx_\phi \cap F$.
		Then
		\[
			C_+ = \{ x \in F \mid \forall \, \sigma \in W : (\sigma x, x_0) \leq 0 \} \, .
		\]
		Let $V = V_1 \oplus V_2$ be the decomposition of $V$ into the irreducible and trivial component.
		Let $x_0 = x_0' + x_0''$ with $x_0' \in V_1$ and $x_0'' \in V_2$.
		By definition $V_2 \not\subset \ker(\phi)$, and so $x_\phi \not\in V_1$, thus $x_0 \not\in V_1$, and hence $x_0'' \neq 0$.
		As $\phi$ is only determined up to multiplication by a positive scalar, suppose without loss of generality that $(x_0'', x_0'') = 1$.
		Let
		\[
			P_+ = \{ x \in F \cap V_1 \mid \forall \, \sigma \in W : (\sigma x, x_0') \leq 1 \} \, .
		\]

		Consider the reflections at the facet hyperplanes of $F$ as simple generators of $W$.
		Let $l$ denote the corresponding length function.
		Let $x \in F \cap V_1$ and $\sigma \in W$.
		Let $H$ be a facet hyperplane of $\sigma F$, such that $l(s \sigma) > l(\sigma)$ for the reflection $s$ at $H$.
		We claim that in this situation
		\begin{equation} \label{oc4mKkWo}
			(s \sigma x, x_0') \leq (\sigma x, x_0') \, .
		\end{equation}
		Indeed, consider the decomposition $V_1 = (H \cap V_1) \oplus H^\perp$.
		According to this decomposition, write $x_0' = v_0 + w_0$ and $\sigma x = v_1 + w_1$;
		then $s \sigma x = v_1 - w_1$.
		We have $(\sigma x, x_0') = (v_1, v_0) + (w_1, w_0)$ and  $(s \sigma x, x_0') = (v_1, v_0) - (w_1, w_0)$.
		Hence
		\begin{equation} \label{zmBvrJsi}
			(s \sigma x, x_0') = (\sigma x, x_0') - 2(w_1,w_0) \, .
		\end{equation}
		Generally, if $\tau \in W$, then $l(\tau)$ equals the number of reflection hyperplanes between $F$ and $\tau F$.
		As $l(s \sigma x) > l(\sigma x)$, this implies that $x_0$ and $\sigma x$ lie on the same side of $H$.
		Hence $(w_1, w_0) \geq 0$, and so \eqref{zmBvrJsi} implies the claim \eqref{oc4mKkWo}.

		By induction on $l(\sigma)$, \eqref{oc4mKkWo} implies that
		\[
			P_+ = \{ x \in F \cap V_1 \mid (x, x_0') \leq 1 \} \, . \]
		The cone $C_+$ is the cone over $P - x_0''$, thus $C_+ = \{ x \in F \mid (x, x_0) \leq 0 \}$.
		The cone $F \cap V_1$ is a fundamental domain for the action of $W$ on $V_1$.
		Hence $F \cap V_1$ is simplicial.
		We have $F = (F \cap V_1) + V_2$, and so $\dim(F) = \dim(F \cap V_1) + 1$.
		The cone $C_+$ is defined in $F$ by the single additional inequality $(x, x_0) \leq 0$, thus $C_+$ is simplicial.
	\end{proof}

	Consider the situation of \autoref{8UQHag7u}.
	Choose an order $H_1, \ldots, H_{n-1}$ of the facet hyperplanes of $F$ with the corresponding simple reflections $s_1, \ldots, s_{n-1}$.
	For any subset $J \subset \{1, \ldots, n-1\}$, let
	\[
		C_J \defeq C_+ \setminus \bigcup_{j \in J} H_j \, .
	\]
	
	For example, $C_\emptyset = C_+$.
	If $J \neq \emptyset$ some of the facets of $C_+$ are removed.
	
	\begin{proposition} \label{JNSFSTk5}
		In the situation of \autoref{8UQHag7u}, $C$ decomposes as a disjoint union
		\[
			C = \bigcup_{\sigma \in W} \sigma \,C_{\Dr(\sigma)} \, .
		\]
	\end{proposition}

	\begin{proof}
		For $x \in C$, let
		\[
			W(x) \defeq \{ \sigma \in W \mid x \in \sigma \, C_+ \} \, .
		\]
		By \autoref{8UQHag7u}, $C = \bigcup_{\sigma \in W} \sigma \, C_+$, and so the set $W(x)$ is nonempty.
		It contains a unique element of minimal length \cite[\S 1.10]{Humphreys}, denoted by $\sigma_x$.

		Assume that $\sigma_x^{-1}x \in H_j$ for some $j \in \Dr(\sigma_x)$.
		Then $s_j \sigma_x^{-1} x = \sigma_x^{-1}x$, and so $x = \sigma_x s_j \sigma_x^{-1} x$.
		As $\sigma_x^{-1} x \in C_+$, this implies that $\sigma_x s_j \in W(x)$.
		On the other hand $l(\sigma_x s_j) < l(\sigma_x)$, a contradiction.
		Hence $\sigma_x^{-1} x \not\in H_j$ for all $j \in \Dr(\sigma_x)$.
		Hence $\sigma_x^{-1} x \in C_{\Dr(\sigma_x)}$, and so $x \in \sigma_x C_{\Dr(\sigma_x)}$.
		This proves that $C = \bigcup_{\sigma \in W} \sigma \, C_{\Dr(\sigma)}$.

		To prove disjointness, let $x \in \sigma \, C_{\Dr(\sigma)}$ for some $\sigma \in W$.
		We have to show that $\sigma = \sigma_x$.
		Clearly $\sigma \in W(x)$.
		It remains to show that $\sigma$ has minimal length in $W(x)$.
		Assume that $\sigma$ has not minimal length in $W(x)$.
		Then there is $j \in \{1, \ldots, n-1\}$ such that $l(\sigma s_j) < l(\sigma)$ and $\sigma s_j \in W(x)$ \cite[\S 1.10]{Humphreys}.
		From $\sigma \in W(x)$	 we conclude that $\sigma^{-1}x \in C_+$ and from $\sigma s_i \in W(x)$ that $s_j \sigma^{-1}x \in C_+$.
		Hence $\sigma^{-1}x \in H_j$.		
		Since $l(\sigma s_j) < l(\sigma)$ we have $j \in \Dr(\sigma)$.
		Hence $\sigma^{-1}x \not\in C_{\Dr(\sigma)}$, and so $x \not\in \sigma \, C_{\Dr(\sigma)}$, a contradiction.
	\end{proof}

\subsection{Generating functions for monoconditional cones}

	Let $V^*_\mathbf{C} = V^* \otimes_\mathbf{R} \mathbf{C}$.
	Extend $\langle\;\;{,}\;\;\rangle$ to $V \times V^*_\mathbf{C}$ by $\mathbf{C}$-linearity in the second argument.	
	Let $\Gamma \subset V$ be a lattice and $S \subset V$.
	Suppose that there is a nonempty open subset $B \subset V^*_\mathbf{C}$ such that the series $\sum_{x \in S \cap \Gamma} e^{- \langle x, \phi \rangle}$ converges for $\phi \in B$ and has a meromorphic continuation to $V^*_\mathbf{C}$.
	We denote this continuation by $f_S$ and call it the \definedterm{generating function} of $S$ with respect to $\Gamma$.

	\begin{example}
		If $C \subset V$ is a cone, let
		\[
			C^\vee \defeq \{ \phi \in V^* \mid \forall \, x \in C : \langle x, \phi \rangle \geq 0 \} \subset V^*
		\]
		be its dual cone.
		The complexified dual of $C$ is defined as
		\[
			C^\vee_\mathbf{C} \defeq \{ \phi \in V^*_\mathbf{C} \mid \forall \, x \in C : \Re(\langle x, \phi \rangle) \geq 0 \} = C^\vee + i \, V^* .
		\]
		Let $C \subset V$ be a salient cone, rational with respect to $\Gamma$.
		Then $\sum_{x \in C \cap \Gamma} e^{- \langle x, \phi \rangle}$ converges on the interior of $C^\vee_\mathbf{C}$ and has a meromorphic continuation $f_C$ to $V^*_\mathbf{C}$.
	\end{example}

	From now on, suppose that $W$ is crystallographic, i.e., that there is a $W$-invariant lattice $\Gamma$ in $V$.
	A full-dimensional simplicial cone $C \subset V$ is called \definedterm{unimodular} (with respect to $\Gamma$) if it is generated by a basis of $\Gamma$. These generators are called \definedterm{primitive}.

	\begin{theorem} \label{vbvXkhNC}
		In the situation of \autoref{8UQHag7u}, suppose that $C_+$ is unimodular with respect to $\Gamma$.
		Let $b_1, \ldots, b_n$ be the primitive generators of $C_+$, enumerated in the unique way such that $b_j \not\in H_j$ for $j \in \{ 1, \ldots, n-1 \}$.
		Then the generating function of $C$ is
		\[
			f_C(\phi) = \sum_{\sigma \in W} \frac{\prod_{j \in \Dr(\sigma)} e^{- \langle \sigma b_j, \phi \rangle }}{(1-e^{- \langle \sigma b_1, \phi \rangle}) \cdots (1-e^{- \langle \sigma b_n, \phi \rangle})} \, .
		\]
	\end{theorem}

	In practice, $\Gamma$ is often endowed with a distinguished basis.
	In this case, it is often more convenient to work with the following formulation.

	\begin{corollary} \label{jx6Xq2Cz}
		In the situation of \autoref{vbvXkhNC}, let $e_1, \ldots, e_n$ be a basis of $\Gamma$.
		Define coordinates $z_j$ on $V^*_\mathbf{C}$ by $z_j(\phi) \defeq e^{- \langle e_j, \phi \rangle}$.
		For $a = a_1 e_1 + \cdots + a_n e_n \in \Gamma$, let $z^a \defeq z_1^{a_1} \cdots z_n^{a_n}$.
		Then
		\[
			f_C = \sum_{\sigma \in W} \frac{\prod_{j \in \Dr(\sigma)} z^{\sigma b_j}}{(1-z^{\sigma b_1}) \cdots (1-z^{\sigma b_n})} \, .
		\]
	\end{corollary}
	
	\begin{proof}[Proof of \autoref{vbvXkhNC}]
		As $C_+$ is unimodular with primitive generators $b_1, \ldots, b_n$, its generating function is
		\[
			f_{C_+}(\phi) = \frac{1}{(1-e^{- \langle b_1, \phi \rangle}) \cdots (1-e^{- \langle b_n, \phi \rangle})}	\, .
		\]
		With unimodularity it also follows that each generator of $C_+$ is only one lattice hyperplane away from the opposite facet.
		Hence
		\[
			C_{\{j\}} \cap \Gamma = (C_+ \setminus H_j) \cap \Gamma = (C_+ \cap \Gamma) + b_j
		\]
		for all $j \in \{1, \ldots, n-1 \}$.
		More generally, $C_{J} \cap \Gamma = (C_+ \cap \Gamma) + \sum_{j \in J} b_j$ for any $J \subset \{1, \ldots, n-1 \}$.
		Applying this observation to $J = \Dr(\sigma)$ for a $\sigma \in W$ and rephrasing it in terms of generating functions, one obtains
		\[
			f_{C_{\Dr(\sigma)}}(\phi) = \frac{\prod_{j \in \Dr(\sigma)} e^{- \langle b_j, \phi \rangle }}{(1-e^{- \langle b_1, \phi \rangle}) \cdots (1-e^{- \langle b_n, \phi \rangle})} \, .
		\]
		Hence for all $\sigma \in W$ it follows that
		\[
			f_{\sigma C_{\Dr(\sigma)}}(\phi) = \frac{\prod_{j \in \Dr(\sigma)} e^{- \langle \sigma b_j, \phi \rangle }}{(1-e^{- \langle \sigma b_1, \phi \rangle}) \cdots (1-e^{- \langle \sigma b_n, \phi \rangle})} \, .
		\]
		By \autoref{JNSFSTk5}, $f_C = \sum_{\sigma \in W} f_{\sigma C_{\Dr(\sigma)}}$, which proves the formula.		
	\end{proof}

\section{Cones with the symmetry of a simplex} \label{AKg4ecYc}

	\autoref{vbvXkhNC} specializes to more concrete identities once we fix a particular almost irreducible reflection group $W$.
	The case of $W$ being the group of symmetries of a simplex has been treated in \cite{BeckGesselLeeSavage}.
	We include this case here to show how the result can be derived from \autoref{vbvXkhNC}.

	Let $S_n$ denote the group of permutations of the set $\{ 1, \ldots, n \}$.
	For $\pi \in S_n$, we define the \definedterm{descent set} of $\pi$ as
	\begin{equation} \label{ZXBAbXGt}
		D(\pi) \defeq \{ j \in \{ 1, \ldots, n-1 \} \mid \pi(j) > \pi(j+1) \} \, .
	\end{equation}
	This is the standard definition used in the literature on permutations.

	The group $S_n$ acts on $\mathbf{R}^n$ by permutation of the components.
	For $\pi \in S_n$, let $\sigma_\pi \in \mathrm{O}(\mathbf{R}^n)$ denote the transformation by which $\pi$ acts on $\mathbf{R}^n$.
	Let $W = \{ \sigma_\pi \mid \pi \in S_n \} \subset \mathrm{O}(\mathbf{R}^n)$.
	Then $W$ is the group of symmetries of the $(n-1)$-dimensional standard simplex.
	For $j = 1, \ldots, n-1$, let $s_j \in W$ be the transposition of the $j$th and $(j+1)$st component in $\mathbf{R}^n$.
	Then $s_1, \ldots, s_{n-1}$ are simple generators of $W$.

	The following shows that the definitions of descent given in \eqref{iz7umNWf} and \eqref{ZXBAbXGt} agree.

	\begin{proposition}[{\cite[Proposition 1.5.3]{BjornerBrenti}}] \label{prop:typeAdes}
	$\Dr(\sigma_\pi) = D(\pi)$ for all $\pi \in S_n$.
	\end{proposition}

        Our main result in this section is the following.

	\begin{proposition}[{\cite[Theorem 1]{BeckGesselLeeSavage}}] \label{BPEevVg3}
		Fix integers $a_1 \leq \cdots \leq a_n$ such that $a_1 + \cdots + a_n = 1$.
		Let
		\[ 
			C \defeq \{ x \in \mathbf{R}^n \mid
			\forall \, \pi \in S_n :
			a_1 x_{\pi(1)} + \cdots + a_n x_{\pi(n)} \geq 0 \} \, .
		\]
		Let $\Sigma_j \defeq a_1 + \cdots + a_j$ for $j \in \{1, \ldots, n-1 \}$.
		The generating function of $C$ with respect to $\mathbf{Z}^n$ is
		\[
			f_C = \frac{1}{1-z_1 \cdots z_n} \sum_{\pi \in S_n} \frac{\prod_{j \in D(\pi)} (z_1 \cdots z_n)^{-\Sigma_j} \prod_{i=1}^j z_{\pi(i)}}{\prod_{j=1}^{n-1} \left(1 - (z_1 \cdots z_n)^{-\Sigma_j} \prod_{i=1}^j z_{\pi(i)} \right)} \, .
		\]
	\end{proposition}

	Note that the condition on the $a_i$ to be increasing is a normalization rather than a restriction.

	\begin{proof}
		The cone $C$ is symmetric and monoconditional for $W$.
		Let $F = \{ x \in \mathbf{R}^n \mid x_1 \geq \cdots \geq x_n \}$, a fundamental domain for $W$.
		Then our chosen simple generators $s_1, \ldots, s_{n-1}$ of $W$ are the reflections at the facet hyperplanes of $F$.
		Let $x_0 = (-a_1, \ldots, -a_n) \in F$.
		By the proof of \autoref{8UQHag7u},
		\begin{align*}
			C_+ 
			&= \{ x \in F \mid (x, x_0) \leq 0 \} = \{ x \in F \mid a_1 x_1 + \cdots + a_n x_n \geq 0 \} \\
			&= \{ x \in \mathbf{R}^n \mid Ax \geq 0 \} \, ,
		\end{align*}
		where
		\[
			A = \begin{pmatrix} \begin{tikzpicture}[x=(-90:1.5em),y=(0:2.5em),dash pattern=on 0pt off 1ex,line cap=round,very thick]
				\draw (1,1) node[fill=white][fill=white]{$1$} -- (5,5) node[fill=white]{$1$};
				\draw (1,2) node[fill=white][fill=white]{$-1$} -- (5,6) node[fill=white]{$-1$};
				\draw (1,3) node[fill=white]{$0$} -- (1,6) node[fill=white]{$0$} -- (4,6) node[fill=white]{$0$} -- cycle;
				\draw (2,1) node[fill=white]{$0$} -- (5,1) node[fill=white]{$0$} -- (5,4) node[fill=white]{$0$} -- cycle;
				\draw (6,1) node[fill=white]{$a_1$} -- (6,6) node[fill=white]{$a_n$};
			\end{tikzpicture} \end{pmatrix} .
		\]
		The determinant of $A$ is $a_1 + \cdots + a_n = 1$, i.e., $A$ is unimodular and so is $C_+$.
		Let $b_1, \ldots, b_n$ be the primitive generators of $C_+$, enumerated in the unique way such that $b_j \notin H_j$ for $j \in \{1, \ldots, n-1 \}$.
		Then by \autoref{jx6Xq2Cz}, the generating function of $C$ is
		\[
			f_C = \sum_{\sigma \in W} \frac{\prod_{j \in \Dr(\sigma)} z^{\sigma b_j}}{(1-z^{\sigma b_1}) \cdots (1-z^{\sigma b_n})} \, .
		\]
		\autoref{BPEevVg3} follows once we describe $\Dr(\sigma)$, $b_j$, and the action of $W$ explicitly.
		The inverse of $A$ is
		\[
			B \defeq A^{-1} = \begin{pmatrix} \begin{tikzpicture}[x=(-90:1.5em),y=(0:2.5em),dash pattern=on 0pt off 1ex,line cap=round,very thick]
				\draw (1,1) node[fill=white]{$\Sigma_1'$} -- (1,4) node[fill=white]{$\Sigma_{n-1}'$} -- (4,4) node[fill=white]{$\Sigma_{n-1}'$} -- cycle;
				\draw (2,1) node[fill=white]{$-\Sigma_1$} -- (5,1) node[fill=white]{$-\Sigma_1$} -- (5,4) node[fill=white]{$-\Sigma_{n-1}$} -- cycle;
				\draw (1,5) node[fill=white]{$1$} -- (5,5) node[fill=white]{$1$};
			\end{tikzpicture} \end{pmatrix} ,
		\]
		where $\Sigma_j' \defeq 1 - \Sigma_j$.
		Then $b_j$ is the $j$th column vector of $B$.
		Let
		\[
			b_{ij} \defeq \begin{cases}
				1 & \text{if } j = n, \\
				1-\Sigma_j & \text{if } i \leq j < n, \\
				-\Sigma_j & \text{if } j < i
			\end{cases}
		\]
		be the $i$th component of $b_j$, i.e., the $(i,j)$th component of $B$.
		As defined above, with $\pi \in S_n$ we associate $\sigma_\pi \in \mathrm{O}(n)$ by $\sigma_\pi e_i = e_{\pi(i)}$.
		Then $W = \{ \sigma_\pi \mid \pi \in S_n \}$ and we have $\Dr(\sigma_\pi) = D(\pi)$.
		Hence
		\begin{align*}
			f_C
			&= \sum_{\sigma \in W} \frac{\prod_{j \in \Dr(\sigma)} z^{\sigma b_j}}{\prod_{j=1}^n (1-z^{\sigma b_j})} \\
			&= \sum_{\pi \in S_n} \frac{\prod_{j \in \Dr(\sigma_\pi)} z^{\sigma_\pi b_j}}{\prod_{j=1}^n (1-z^{\sigma_\pi b_j})} \\
			&= \sum_{\pi \in S_n} \frac{\prod_{j \in D(\pi)} \prod_{i=1}^n z_{\pi(i)}^{b_{ij}}}{\prod_{j=1}^n \left(1 - \prod_{i=1}^n z_{\pi(i)}^{b_{ij}} \right)} \\
			&= \frac{1}{1-z_1 \cdots z_n} \sum_{\pi \in S_n} \frac{\prod_{j \in D(\pi)} (z_1 \cdots z_n)^{-\Sigma_j} \prod_{i=1}^j z_{\pi(i)}}{\prod_{j=1}^{n-1} \left(1 - (z_1 \cdots z_n)^{-\Sigma_j} \prod_{i=1}^j z_{\pi(i)} \right)} \, .
			\qedhere
		\end{align*}
	\end{proof}

\section{Cones with hyperoctahedral symmetry} \label{dxuwjrHR}

	We now consider the case of cones which are symmetric under the action of a hyperoctahedral group.
	Let $W \subset \mathrm{O}(n)$ be the hyperoctahedral group on the first $n-1$ components of $\mathbf{R}^n$.
	Let $s_1 \in W$ be the sign change in the first component and, for $j = 2, \ldots, n-1$, let $s_j \in W$ be the transposition of the $(j-1)$st and $j$th component in $\mathbf{R}^n$.
	Then $s_1, \ldots, s_{n-1}$ are simple generators of $W$.

	For combinatorial (as opposed to geometric) arguments, it is often more convenient to use the following parameterization of the hyperoctahedral group:
	For $\pi \in S_{n-1}$ and $\epsilon \in \{\pm 1\}^{n-1}$, define $\sigma_{\pi,\epsilon} \in \mathrm{O}(n)$ by
	\begin{equation} \label{CzP35kqx}
		\sigma_{\pi,\epsilon}e_i = \epsilon_i e_{\pi(i)} \, ,
	\end{equation}
	where we use the convention that $\pi(n) \defeq n$ for $\pi \in S_{n-1}$ and $\epsilon_n \defeq 1$ for $\epsilon \in \{\pm 1\}^{n-1}$.
	Then $W = \{ \sigma_{\pi,\epsilon} \mid \pi \in S_{n-1},$ $\epsilon \in \{\pm 1\}^{n-1} \}$.
	Let $B_{n-1}$ denote the set $S_{n-1} \times \{\pm 1\}^{n-1}$,
	endowed with the group structure such that $\sigma : B_{n-1} \to W$ becomes an isomorphism of groups.

	In terms of this parameterization, the right descent set of $W$ can be expressed more explicitly.
	For $(\pi, \epsilon) \in B_{n-1}$ let
	\begin{equation} \label{82aH5Zd5}
		D(\pi, \epsilon) \defeq \{ j \in \{1, \ldots, n-1\} \mid \epsilon_{j-1} \pi(j-1) > \epsilon_j \pi(j) \}
	\end{equation}
	with the convention that $\epsilon_0 \pi(0) \defeq 0$.
	Then the following holds.

	\begin{proposition}[{\cite[Proposition 8.1.2]{BjornerBrenti}}] \label{prop:typeBdes}
		For all $(\pi, \epsilon) \in B_{n-1}$, we have
		\[
			\Dr(\sigma_{\pi,\epsilon}) = D(\pi,\epsilon) \, .
		\]
	\end{proposition}

	Note that the descent set defined in \eqref{82aH5Zd5} is translated by $+1$ with respect to definitions found in the literature on signed permutations.
	This is because to have a consistent setup in \autoref{sH4iHn8E}, we always start the enumeration of the simple reflections with $1$, whereas from a signed permutations perspective it is convenient to start this enumeration with $0$.

	We define the descent statistic on the hyperoctahedral group by setting the \definedterm{descent number}
	\[
		\des(\pi,\epsilon)\defeq \lvert D(\pi,\epsilon) \rvert
	\]
	for $(\pi, \epsilon) \in B_{n-1}$.
	Similarly, the \definedterm{major index} is
	\[
		\maj(\pi,\epsilon)\defeq\sum_{j \in D(\pi,\epsilon)} (j-1)
	\]
	and the \definedterm{comajor index} is
	\[
		\comaj(\pi,\epsilon)
		\defeq \sum_{j \in D(\pi,\epsilon)} (n-j)
	\]
	for $(\pi, \epsilon) \in B_{n-1}$.
	It follows that we have the relationship
	\begin{equation}\label{eqn:comajdesmaj}
		\comaj(\pi,\epsilon)=(n-1)\des(\pi,\epsilon)-\maj(\pi,\epsilon) \, .
	\end{equation}

	\subsection{The multivariate generating function}

	In this situation, \autoref{jx6Xq2Cz} specializes as follows.

	\begin{proposition} \label{BDjVGd4o}
		Fix integers $0 \leq a_1 \leq \cdots \leq a_{n-1} \neq 0$.
		Let
		\[ \begin{split}
			C \defeq \{ x \in \mathbf{R}^n \mid {}
			& \forall \, \pi \in S_{n-1},\ \epsilon \in \{\pm 1\}^{n-1} : \\
			& \epsilon_1 a_1 x_{\pi(1)} + \cdots + \epsilon_{n-1} a_{n-1} x_{\pi(n-1)} \leq x_n \} \, .
		\end{split} \]
		
		The generating function of $C$ with respect to $\mathbf{Z}^n$ is
		\[
			f_C = \frac{1}{1-z_n} \sum_{\pi \in S_{n-1}} \sum_{\epsilon \in \{\pm 1\}^{n-1}} \frac{\prod_{j \in D(\pi, \epsilon)} \prod_{i=j}^{n-1} z_{\pi(i)}^{\epsilon_i} z_n^{a_i}}{\prod_{j=1}^{n-1} \left(1 - \prod_{i=j}^{n-1} z_{\pi(i)}^{\epsilon_i} z_n^{a_i} \right)} \, .
		\]
	\end{proposition}
	
	Note that the condition on the $a_i$ to be nonnegative and increasing is a normalization rather than a restriction.	

	\begin{proof}
		The cone $C$ is symmetric and monoconditional for $W$.
		Let 
		\[
			F \defeq \{ x \in \mathbf{R}^n \mid 0 \leq x_1 \leq \cdots \leq x_{n-1} \} \, ,
		\]
		a fundamental domain for $W$.
		The $s_1, \ldots, s_{n-1}$ defined previously are the simple generators of $W$ corresponding to $F$.
		Let $x_0 \defeq (a_1, \ldots, a_{n-1}, -1) \in F$.
		By the proof of \autoref{8UQHag7u},
		\begin{align*}
			C_+
			&= \{ x \in F \mid (x, x_0) \leq 0 \}
			= \{ x \in F \mid a_1 x_1 + \cdots + a_{n-1} x_{n-1} \leq x_n \} \\
			&= \{ x \in \mathbf{R}^n \mid Ax \geq 0 \} \, ,
		\end{align*}
		where
		\[
			A = \begin{pmatrix} \begin{tikzpicture}[x=(-90:1.5em),y=(0:2.5em),dash pattern=on 0pt off 1ex,line cap=round,very thick]
				\draw (1,1) node[fill=white][fill=white]{$1$} -- (5,5) node[fill=white]{$1$};
				\draw (2,1) node[fill=white]{$-1$} -- (5,4) node[fill=white]{$-1$};
				\draw (1,2) node[fill=white]{$0$} -- (1,6) node[fill=white]{$0$} -- (5,6) node[fill=white]{$0$} -- cycle;
				\draw (3,1) node[fill=white]{$0$} -- (5,1) node[fill=white]{$0$} -- (5,3) node[fill=white]{$0$} -- cycle;
				\draw (6,1) node[fill=white]{$-a_1$} -- (6,5) node[fill=white]{$-a_{n-1}$};
				\draw (6,6) node[fill=white]{$1$};
			\end{tikzpicture} \end{pmatrix} .
		\]
		The matrix $A$ and hence $C_+$ is unimodular.
		Let $b_1, \ldots, b_n$ be the primitive generators of $C_+$, enumerated in the unique way such that $b_j \notin H_j$ for $j < n$.
		Then by \autoref{jx6Xq2Cz}, the generating function of $C$ is
		\[
			f_C = \sum_{\sigma \in W} \frac{\prod_{j \in \Dr(\sigma)} z^{\sigma b_j}}{(1-z^{\sigma b_1}) \cdots (1-z^{\sigma b_n})} \, .
		\]

		The inverse of $A$ is
		\[
			B \defeq A^{-1} = \begin{pmatrix} \begin{tikzpicture}[x=(-90:1.5em),y=(0:2.5em),dash pattern=on 0pt off 1ex,line cap=round,very thick]
				\draw (1,1) node[fill=white]{$1$} -- (4,1) node[fill=white]{$1$} -- (4,4) node[fill=white]{$1$} -- cycle;
				\draw (1,2) node[fill=white]{$0$} -- (1,5) node[fill=white]{$0$} -- (4,5) node[fill=white]{$0$} -- cycle;
				\draw (5,1) node[fill=white]{$\Sigma_1$} -- (5,4) node[fill=white]{$\Sigma_{n-1}$};
				\draw (5,5) node[fill=white]{$1$};
			\end{tikzpicture} \end{pmatrix} ,
		\]
		where $\Sigma_j \defeq a_j + \cdots + a_{n-1}$.
		Then $b_j$ is the $j$th column vector of $B$.
		Let
		\[
			b_{ij} \defeq \begin{cases}
				0 & \text{if } i < j , \\
				1 & \text{if } j \leq i < n \text{ or } i = j = n , \\
				\Sigma_j & \text{if } j < i = n
			\end{cases}
		\]
		be the $i$th component of $b_j$, i.e., the $(i,j)$th component of $B$.
		By \autoref{prop:typeBdes} and using our notation introduced at the beginning of this section,
		\begin{align*}
			f_C
			&= \sum_{\sigma \in W} \frac{\prod_{j \in \Dr(\sigma)} z^{\sigma b_j}}{\prod_{j=1}^n (1-z^{\sigma b_j})} \\
			&= \sum_{(\pi, \epsilon) \in B_{n-1}} \frac{\prod_{j \in \Dr(\sigma_{\pi,\epsilon})} z^{\sigma_{\pi,\epsilon} b_j}}{\prod_{j=1}^n (1-z^{\sigma_{\pi,\epsilon} b_j})} \\
			&= \sum_{(\pi, \epsilon) \in B_{n-1}} \frac{\prod_{j \in D(\pi, \epsilon)} \prod_{i=1}^n z_{\pi(i)}^{\epsilon_i b_{ij}}}{\prod_{j=1}^n \left(1 - \prod_{i=1}^n z_{\pi(i)}^{\epsilon_i b_{ij}} \right)} \\
			&= \frac{1}{1-z_n} \sum_{(\pi, \epsilon) \in B_{n-1}} \frac{\prod_{j \in D(\pi, \epsilon)} \prod_{i=j}^{n-1} z_{\pi(i)}^{\epsilon_i} z_n^{a_i}}{\prod_{j=1}^{n-1} \left(1 - \prod_{i=j}^{n-1} z_{\pi(i)}^{\epsilon_i} z_n^{a_i} \right)} \, .
			\qedhere
		\end{align*}
	\end{proof}

\subsection{Hyperoctahedral Eulerian polynomials} \label{uCUseF9t}

	In the remainder of \autoref{dxuwjrHR}, we provide applications of \autoref{BDjVGd4o} with connections to permutation statistics and Ehrhart theory.
	Our first application is well known, going back to \cite{brentieulerian} and \cite{steingrimsson}; the polyhedral perspective of the following identity was first established in \cite{steingrimsson}, also using Ehrhart theory.

	\begin{corollary}[\cite{brentieulerian}, \cite{steingrimsson}] \label{j2dFgCe6}
		The hyperoctahedral Eulerian polynomials are given by
		\[
			\sum_{(\pi, \epsilon) \in B_{n-1}} t^{\des(\pi,\epsilon)}
			= (1-t)^n \sum_{k=0}^\infty (2k+1)^{n-1} t^k \, .
		\]
	\end{corollary}

	\begin{proof}
		Let
		\[
			P = [-1,1]^{n-1} \, .
		\]
		be the $(n-1)$-dimensional hypercube.
		Our strategy to prove \autoref{j2dFgCe6} is to compute the Ehrhart series		\[
			\Ehr_P(t) \defeq \sum_{k \geq 0} \lvert kP \cap \mathbf{Z}^{n-1} \rvert \cdot t^k
		\]
		of $P$ in two different ways and to conclude by comparing the results.

		On the one hand, note that the cone $C$ over $P$,
		\[
			C = \{ x \in \mathbf{R}^n \mid \forall j < n : \lvert x_j \rvert \leq x_n \} \, ,
		\]
		is the cone considered in \autoref{BDjVGd4o} for $a_1 = \cdots = a_{n-2} = 0$, $a_{n-1} = 1$, so by \autoref{BDjVGd4o} its generating function is
		\[
		f_C = \frac{1}{1-z_n} \sum_{(\pi, \epsilon) \in B_{n-1}} \frac{\prod_{j \in D(\pi, \epsilon)} \left( z_n \prod_{i=j}^{n-1} z_{\pi(i)}^{\epsilon_i} \right)}{\prod_{j=1}^{n-1} \left(1 - z_n \prod_{i=j}^{n-1} z_{\pi(i)}^{\epsilon_i}  \right)} \, .
		\]
		Since $\Ehr_P(t)$ is obtained by evaluating $f_C$ at $z_1 = \cdots = z_{n-1} = 1$, $z_n = t$, we obtain
		\begin{align*}
		\Ehr_P(t)
		&= \frac{1}{1-t} \sum_{(\pi, \epsilon) \in B_{n-1}} \frac{\prod_{j \in D(\pi, \epsilon)} t}{\prod_{j=1}^{n-1} (1 - t)} \\
		&= \frac{1}{(1-t)^n} \sum_{(\pi, \epsilon) \in B_{n-1}} t^{\des(\pi,\epsilon)} \, .
		\end{align*}
		On the other hand, by definition
		\[
			\Ehr_P(t)
			= \sum_{k \geq 0} (2k+1)^{n-1} t^k \, .
		\]
		Together, we obtain
		\[
			\frac{1}{(1-t)^n} \sum_{(\pi, \epsilon) \in B_{n-1}} t^{\des(\pi,\epsilon)}
			= \sum_{k \geq 0} (2k+1)^{n-1} t^k
		\]
		and \autoref{j2dFgCe6} follows.
	\end{proof}

\subsection{The distribution of the comajor index} \label{TkEY25r3}

	We show here how to derive the distribution of the comajor index; while this is likely well known, we could not find an explicit statement in the literature.
	For $k \in \mathbf{N}$ and a variable $t$, let
	\[
		[k]_t \defeq 1 + t + t^2 + \cdots + t^{k-1}
		\qquad \text{ and } \qquad
		[k]_t! \defeq [1]_t [2]_t \cdots [k]_t \, .
	\]

	\begin{corollary} \label{JtP8XU22}
		The distribution of the comajor index on the hyperoctahedral group is given by
		\[
			\sum_{(\pi, \epsilon) \in B_{n-1}} t^{\comaj(\pi,\epsilon)}
			= (1+t)^{n-1} \, [n-1]_t! \, .
		\]
	\end{corollary}

	\begin{proof}
		Let
		\[
			P = \{ x \in \mathbf{R}^{n-1} \mid \lvert x_1 \rvert + \cdots + \lvert x_{n-1} \rvert \leq 1 \}
		\]
		be the $(n-1)$-dimensional cross-polytope.
		Our strategy to prove \autoref{JtP8XU22} is to compute the Ehrhart series
		\[
			\Ehr_P(t) \defeq \sum_{k \geq 0} \lvert kP \cap \mathbf{Z}^{n-1} \rvert \cdot t^k
		\]
		of $P$ in two different ways and to conclude by comparing the results.
		
		On the one hand, note that the cone $C$ over $P$,
		\[
			C = \{ x \in \mathbf{R}^n \mid \lvert x_1 \rvert + \cdots + \lvert x_{n-1} \rvert \leq x_n \} \, ,
		\]
		is the cone considered in \autoref{BDjVGd4o} for $a_1 = \cdots = a_{n-1} = 1$, so by \autoref{BDjVGd4o} its generating function is
		\[
		f_C = \frac{1}{1-z_n} \sum_{(\pi, \epsilon) \in B_{n-1}} \frac{\prod_{j \in D(\pi, \epsilon)} \prod_{i=j}^{n-1} z_{\pi(i)}^{\epsilon_i} z_n}{\prod_{j=1}^{n-1} \left(1 - \prod_{i=j}^{n-1} z_{\pi(i)}^{\epsilon_i} z_n \right)} \, .
		\]
		Since $\Ehr_P(t)$ is obtained by evaluating $f_C$ at $z_1 = \cdots = z_{n-1} = 1$, $z_n = t$, we obtain
		\begin{align*}
		\Ehr_P(t)
		&= \frac{1}{1-t} \sum_{(\pi, \epsilon) \in B_{n-1}} \frac{\prod_{j \in D(\pi, \epsilon)} \prod_{i=j}^{n-1} t}{\prod_{j=1}^{n-1} \left(1 - \prod_{i=j}^{n-1} t \right)} \\
		&= \frac{\sum_{(\pi, \epsilon) \in B_{n-1}} \prod_{j \in D(\pi, \epsilon)} t^{n-j}}{(1-t) \prod_{j=1}^{n-1} \left(1 - t^{n-j} \right)} \\
		&= \frac{\sum_{(\pi, \epsilon) \in B_{n-1}} t^{\comaj(\pi,\epsilon)}}{(1-t) \prod_{j=1}^{n-1} \left(1 - t^{n-j} \right)} \, .
		\end{align*}
		On the other hand, it is known \cite[Theorem 2.7]{BeckRobins} that
		\[
			\Ehr_P(t)
			= \frac{(1+t)^{n-1}}{(1-t)^n} \, .
		\]
		Together, we obtain
		\[
			\frac{\sum_{(\pi, \epsilon) \in B_{n-1}} t^{\comaj(\pi,\epsilon)}}{(1-t) \prod_{j=1}^{n-1} \left(1 - t^{n-j} \right)}
			= \frac{(1+t)^{n-1}}{(1-t)^n} \, ,
		\]
		so
		\begin{align*}
			\sum_{(\pi, \epsilon) \in B_{n-1}} t^{\comaj(\pi,\epsilon)}
			&= \frac{(1+t)^{n-1} (1-t) \prod_{j=1}^{n-1} (1 - t^{n-j})}{(1-t)^n} \\
			&= \frac{(1+t)^{n-1} \prod_{j=1}^{n-1} (1 - t^{n-j})}{(1-t)^{n-1}} \\
			&= (1+t)^{n-1} \prod_{j=1}^{n-1} \frac{1 - t^{n - j}}{1 - t} \\
			&= (1+t)^{n-1} \, [n-1]_t! \, . \qedhere
		\end{align*}
	\end{proof}

	\begin{remark}
		The distributions for the descent and comajor index statistics on $B_n$ arise from studying simple choices of the $a_i$ from the set of $0/1$-vectors.
		It would be interesting to determine the structure of the multivariate generating functions (or their specializations) when other $0/1$-vectors are used; due to the hyperoctahedral symmetry of our cones, this amounts to studying the case
		\[
			(a_1,a_2,\ldots,a_{n-1})=(0,\ldots,0,1,\ldots,1) \, .
		\]
		The resulting cones interpolate naturally between cones over hypercubes and cones over crosspolytopes.
		While we were not able to treat this family of polytopes using the methods exposed in this article, the following section shows how to do so for a different interpolation.
	\end{remark}

\subsection{Almost constant coefficients} 

	In this section, we show how to use \autoref{BDjVGd4o} to obtain a closed form expression for the Ehrhart series of a family of rational polytopes interpolating between the hypercube and the cross polytope (considered in subsections~\ref{uCUseF9t} and \ref{TkEY25r3}, respectively).
	These are the polytopes $P$ such that the cone over $P$ is of the form considered in \autoref{BDjVGd4o} such that $a_1$ through $a_{n-2}$ coincide.
	An example of such a polytope is shown in \autoref{cQJGQeJR}.

	\begin{figure}
		\centering
		\begin{tikzpicture}[x=(0:2em),y=(90:2em),z=(-140:.8em),join=round]
			\draw[gray] (-6,0,0) -- (-4,-4,0) -- (-3,-3,-3) -- (-4,0,-4) -- cycle;
			\draw[gray] (0,-6,0) -- (0,-4,-4) -- (-3,-3,-3) -- (-4,-4,0) -- cycle;
			\draw[gray] (0,0,-6) -- (-4,0,-4) -- (-3,-3,-3) -- (0,-4,-4) -- cycle;
			\draw[gray] (0,-6,0) -- (0,-4,-4) -- (3,-3,-3) -- (4,-4,0) -- cycle;
			\draw[gray] (6,0,0) -- (4,-4,0) -- (3,-3,-3) -- (4,0,-4) -- cycle;
			\draw[gray] (0,0,-6) -- (4,0,-4) -- (3,-3,-3) -- (0,-4,-4) -- cycle;
			\draw[gray] (-6,0,0) -- (-4,4,0) -- (-3,3,-3) -- (-4,0,-4) -- cycle;
			\draw[gray] (0,6,0) -- (0,4,-4) -- (-3,3,-3) -- (-4,4,0) -- cycle;
			\draw[gray] (0,0,-6) -- (-4,0,-4) -- (-3,3,-3) -- (0,4,-4) -- cycle;
			\draw[gray] (6,0,0) -- (4,4,0) -- (3,3,-3) -- (4,0,-4) -- cycle;
			\draw[gray] (0,6,0) -- (0,4,-4) -- (3,3,-3) -- (4,4,0) -- cycle;
			\draw[gray] (0,0,-6) -- (4,0,-4) -- (3,3,-3) -- (0,4,-4) -- cycle;
			\draw (-6,0,0) -- (-4,-4,0) -- (-3,-3,3) -- (-4,0,4) -- cycle;
			\draw (0,-6,0) -- (0,-4,4) -- (-3,-3,3) -- (-4,-4,0) -- cycle;
			\draw (0,0,6) -- (-4,0,4) -- (-3,-3,3) -- (0,-4,4) -- cycle;
			\draw (0,-6,0) -- (0,-4,4) -- (3,-3,3) -- (4,-4,0) -- cycle;
			\draw (6,0,0) -- (4,-4,0) -- (3,-3,3) -- (4,0,4) -- cycle;
			\draw (0,0,6) -- (4,0,4) -- (3,-3,3) -- (0,-4,4) -- cycle;
			\draw (-6,0,0) -- (-4,4,0) -- (-3,3,3) -- (-4,0,4) -- cycle;
			\draw (0,6,0) -- (0,4,4) -- (-3,3,3) -- (-4,4,0) -- cycle;
			\draw (0,0,6) -- (-4,0,4) -- (-3,3,3) -- (0,4,4) -- cycle;
			\draw (6,0,0) -- (4,4,0) -- (3,3,3) -- (4,0,4) -- cycle;
			\draw (0,6,0) -- (0,4,4) -- (3,3,3) -- (4,4,0) -- cycle;
			\draw (0,0,6) -- (4,0,4) -- (3,3,3) -- (0,4,4) -- cycle;
		\end{tikzpicture}
		\caption{The rational polytope $P = \{ x \in \mathbf{R}^3 \mid \lvert x_1 \rvert + \lvert x_2 \rvert + \lvert x_3 \rvert + \max {\{ \lvert x_1 \rvert, \lvert x_2 \rvert, \lvert x_3 \rvert \}} \leq 1 \}$.
		Its Ehrhart series can be computed by \autoref{32siiSSZ}.}
		\label{cQJGQeJR}
	\end{figure}
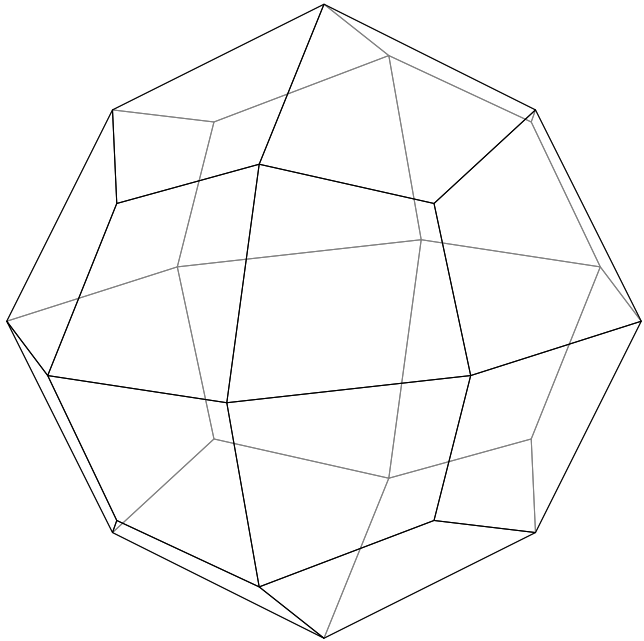

	\begin{corollary} \label{32siiSSZ}
		Let $b, c \geq 0$, not both $0$.
		Let
		\[
			P =
			\{
				x \in \mathbf{R}^{n-1}
			\mid
				c \cdot (\lvert x_1 \rvert + \cdots +  \lvert x_{n-1} \rvert)
				+ b\cdot \max {\{ \lvert x_1 \rvert, \ldots, \lvert x_{n-1} \rvert \}}
				\leq 1
			\} \, .
		\]
		Then the Ehrhart series of $P$ is
		\[
			\Ehr_P(t) = \begin{cases}
				[c]_t(1+t^c)^{n-1}/(1-t^c)^n
				& \text{if $b = 0$,} \\
				[b]_t \sum_{k \geq 0}\left([k+1]_{t^c}+t^c[k]_{t^c}\right)^{n-1} t^{bk}
				& \text{if $b \geq 1$.}
			\end{cases}
		\]
	\end{corollary}

	\begin{proof}
		In \autoref{BDjVGd4o}, we set $a_1= \cdots = a_{n-2}=c$ and $a_{n-1}=c+b$.
		Then $\Ehr_P(t) = f_C(t) \defeq f_C(1, \ldots, 1, t)$, the generating function of $C$ evaluated at $z_1, \ldots, z_{n-1} = 1$, $z_n = t$.
		
		For $b > 0$, the generating function $f_C(t)$ becomes
		\begin{align*}
			f_C(t) = & \frac{\sum_{(\pi,\epsilon) \in B_{n-1}}\prod_{j \in D(\pi,\epsilon)}t^{c(n-j)+b}}
			{(1-t)\prod_{j=1}^{n-1}(1-t^{c(n-j)+b})}\\
			= & \frac{\sum_{(\pi,\epsilon) \in B_{n-1}}(t^{c})^{\comaj(\pi,\epsilon)}(t^b)^{\des(\pi,\epsilon)}}
			{(1-t)\prod_{j=1}^{n-1}(1-(t^c)^jt^b)}.
		\end{align*}
		For $c\geq 1$, if $b=0$ then $f_C(t) = f_{P}(t^c)[c]_t$ where $P$ is the $(n-1)$-dimensional crosspolytope discussed in \autoref{TkEY25r3}; this can be seen by direct computation using \autoref{BDjVGd4o}, and is also a consequence of the fact that crosspolytopes are reflexive \cite[Chapter 4]{BeckRobins}.
		Otherwise, to simplify, we make use of a result of Chow and Gessel \cite[eq.\ (26)]{ChowGessel} to compute the joint distribution of descent and comajor index over $B_n$, namely
		\begin{equation}\label{chowgess}
			\sum_{k \geq 0}([k+1]_q+[k]_q)^n x^k =
			\frac{\sum_{(\pi,\epsilon) \in B_{n}} x^{\des(\pi,\epsilon)}
			q^{\maj(\pi,\epsilon)}}
			{\prod_{j=0}^{n}(1-xq^i)}.
		\end{equation}
		Observe that for $(\pi,\epsilon) \in B_n$, using \eqref{eqn:comajdesmaj},
		\[
			q^{\comaj(\pi,\epsilon)} = 
			(q^n)^{\des(\pi,\epsilon)}
			({1}/{q})^{\maj(\pi,\epsilon)}.
		\]
		Thus, substituting into \eqref{chowgess}, we get
		\begin{align*}
			\sum_{(\pi,\epsilon) \in B_{n}} x^{\des(\pi,\epsilon)}
			q^{\comaj(\pi,\epsilon)}  =&
			\sum_{(\pi,\epsilon) \in B_{n}} (xq^n)^{\des(\pi,\epsilon)}
			(1/q)^{\maj(\pi,\epsilon)}\\
			  =&
			\prod_{i=0}^n (1-xq^{n-i})
			\sum_{k \geq 0}([k+1]_{1/q}+[k]_{1/q})^n (xq^n)^k\\
			=&
			\prod_{i=0}^n (1-xq^{i})
			\sum_{k \geq 0}([k+1]_{q}+q[k]_{q})^n (x)^k .
		\end{align*}
		To get the numerator of $f_C(t)$, we set
		$n=n-1$, $x=t^b$ and $q=t^c$ in the last line above and get:
		\begin{align*}
			f_C(t) = &
			\frac{\prod_{i=0}^{n-1} (1-t^{ci+b})
			\sum_{k \geq 0}([k+1]_{t^c}+t^c[k]_{t^c})^{n-1} t^{bk}}
			{(1-t)\prod_{j=1}^{n-1}(1-t^{cj+b})}\\
			= & [b]_t \sum_{k \geq 0}([k+1]_{t^c}+t^c[k]_{t^c})^{n-1} t^{bk}. \qedhere
		\end{align*}
	\end{proof}

\subsection{Coefficients in arithmetic progression and lecture hall partitions} \label{rjvitVcP}

	If we further generalize the results of the previous subsections to allow $a_i$ to be a linear function of $i$, then $f_C(t)$ can be expressed in terms of lecture hall partitions.

	Lecture hall partitions, introduced by Bousquet-M\'elou and Eriksson \cite{BME1}, are elements of the set
	\[
		L_n
		= \{ \lambda \in \mathbf{Z}^n \mid   0 \leq \tfrac{\lambda_1}{1} \leq \tfrac{\lambda_2}{2} \leq \cdots \leq \tfrac{\lambda_n}{n} \} \, . 
	\]

	The following relationship between statistics on lecture hall partitions and statistics on signed permutations follows from work of Pensyl and Savage \cite{PS2}:
	\begin{equation}\label{eqn:ps}
		\sum_{\lambda \in L_n}
		x^{\left\lceil \frac{\lambda_n}{2n} \right\rceil}
		q^{\stat_1(\lambda)}
		y^{\stat_2(\lambda)}
		=
	 	\frac
		{\sum_{(\pi,\epsilon) \in B_{n}}q^{\comaj(\pi,\epsilon)}x^{\des(\pi,\epsilon)}(y^2)^{\cobin(\pi,\epsilon)}}
		{\prod_{i=0}^{n-1}(1-xq^{n-i}y^{2((i+1) + \cdots + n)})}
		\, ,
	\end{equation}
	where
	\[
		\stat_1(\lambda)
		= \sum_{i=1}^n \left\lceil \frac{\lambda_i}{2i} \right\rceil
		,	 \quad
		\stat_2(\lambda)
		= \sum_{i=1}^n 2i \left\lceil \frac{\lambda_i}{2i} \right\rceil
		,
	\]
	\[
		\cobin(\pi,\epsilon)
		= \sum_{j \in D(\pi,\epsilon)} (j + \cdots + n)
	\]
	for $\lambda \in L_n$, $(\pi,\epsilon) \in B_n$, and $\lceil x \rceil = \inf([x, \infty) \cap \mathbf{Z})$.
	
	We can apply \autoref{BDjVGd4o}, with $z_1= \cdots = z_{n-1}=1$, $z_n=t$ and appropriate choices of $a_i$, to establish a surprising connection between lecture hall partitions and type $B$ symmetrically constrained cones.
	
	\begin{corollary}
		Let $d \geq 0$, $c \geq -2d$, and $b \geq 0$, not all $0$.
		Define $a_1, \ldots, a_{n-1}$ by
		\[
			a_i = 2di + c
			\quad
			(i = 1,  \ldots, n-2),
			\quad
			a_{n-1}=2d(n-1)+c+b
			\, .
		\]
		Then
		\[
			f_C(t)
			=
			\frac{1}{1-t}
			\sum_{\lambda \in L_{n-1}}
			t^{\sum_{i=1}^{n-1} a_i \left\lceil \frac{\lambda_{i}}{2i} \right\rceil}
			\, .
		\]
	\end{corollary}
	
	\begin{proof}
		Observe that for $1 \leq i \leq n-2$,
		substituting the values of $a_i$ into \autoref{BDjVGd4o} with $z_1= \cdots = z_{n-1}=1$, $z_n=t$,
		and then using
		\eqref{eqn:ps}, we obtain
		\begin{align*}
			f_C(t)
			&= 
			\frac{\sum_{(\pi,\epsilon) \in B_{n-1}} \prod_{j \in D(\pi,\epsilon)} t^{a_j+ \cdots + a_{n-1}}}
			{(1-t)\prod_{j=1}^{n-1}(1-t^{a_j + \cdots + a_{n-1}})}
			\\ &=
			\frac{\sum_{(\pi,\epsilon) \in B_{n-1}} \prod_{j \in D(\pi,\epsilon)} t^{b+(n-j)c+2d(j + \cdots + (n-1))}}
			{(1-t)\prod_{i=1}^{n-1}(1-t^{(b+c(n-i)+2d(i + \cdots + (n-1))}}
			\\ &=
			\frac
			{\sum_{(\pi,\epsilon) \in B_{n-1}} (t^c)^{\comaj(\pi,\epsilon)}(t^b)^{\des(\pi,\epsilon)}(t^{2d})^{\cobin(\pi,\epsilon)}}
			{(1-t)\prod_{i=1}^{n-1}(1-t^{(b+c(n-i)+2d(i + \cdots + (n-1))}}
			\\ &=
			\frac
			{\sum_{(\pi,\epsilon) \in B_{n-1}}(t^c)^{\comaj(\pi,\epsilon)}(t^b)^{\des(\pi,\epsilon)}(t^{2d})^{\cobin(\pi,\epsilon)}}
			{(1-t)\prod_{i=1}^{n-1}(1-t^{(b+c(n-i)+2d(i + \cdots + (n-1))}}
			\\ &=
			\frac{1}{1-t}
			\sum_{\lambda \in L_{n-1}}
			t^{b \left \lceil \frac{\lambda_{n-1}}{2n-2} \right\rceil + c \stat_1(\lambda) + d \stat_2(\lambda)}
			\\ &=
			\frac{1}{1-t}
			\sum_{\lambda \in L_{n-1}}
			t^{\sum_{i=1}^{n-1} a_i \left\lceil \frac{\lambda_{i}}{2i} \right\rceil} \, . \qedhere
		\end{align*}
	\end{proof}
	
	As an example, let $n=3$, $a_1=2$ and $a_2=4$.
	Then
	\[
		C = \{ x \in \mathbf{R}^3 \mid {}
		\forall \, \pi \in S_{2},\ \epsilon \in \{\pm 1 \}^2 :
		2 \epsilon_1 x_{\pi(1)}  + 4 \epsilon_{2} x_{\pi(2)} \leq x_3 \} \, ,
	\]
	and from \autoref{BDjVGd4o},
	\begin{align*}
		\sum_{x \in C} t^{x_3}
		&= \frac{1+3t^3+3t^6+t^{10}}{(1-t)(1-t^6)(1-t^4)} \\
		&= 1+t+t^2+t^3+5t^4+5t^5+9t^6+9t^7+13t^8+13t^9+\cdots.
	\end{align*}
	On the other hand, checking the corollary, we have
	\begin{align*}
		\frac{1}{1-t} \sum_{\lambda \in L_2} t^{2 \left\lceil \frac{\lambda_1}{2} \right\rceil + 4 \left\lceil \frac{\lambda_2}{4} \right\rceil}
		&= \frac{1+4t^4+4t^6+4t^8+8t^{10}+8t^{12}+8t^{14} + \cdots}{1-t}\\
		&= \begin{aligned}[t]&1+t+t^2+t^3+5t^4+5t^5+9t^6\\&{+}\;9t^7+13t^8+13t^9+\cdots.\end{aligned}
	\end{align*}

\section{Cones with symmetry of type $D$} \label{XC4vSQKD}

	In this section, we consider the case of monoconditional cones with symmetry given by a Coxeter group of type $D$.
	Unsurprisingly, much of the setup in this section is similar to the hyperoctahedral case; the most notable new feature is that we consider lattice point enumeration with respect to a sublattice of the standard integer lattice.

	Throughout this section, let $W$ be the finite reflection group of type $D_{n-1}$ on the first $n-1$ components of $\mathbf{R}^n$.
	Specifically, let $s_1 \in \mathrm{O}(n)$ be the reflection at the hyperplane $\{ x \in \mathbf{R}^n \mid x_1 + x_2 = 0 \}$.
	For $j = 2, \ldots, n-1$, let $s_j \in \mathrm{O}(n)$ be the transposition of the $(j-1)$st and $j$th component in $\mathbf{R}^n$.
	Then $s_1, \ldots, s_{n-1}$ are the simple generators of $W$.

	 We next describe $\Dr(\sigma)$ and the action of $W$ explicitly.
	Let
	\[
		E_{n-1} \defeq \{ \epsilon \in \{\pm 1\}^{n-1} \mid \epsilon_1 \cdots \epsilon_{n-1} = 1 \} \, .
	\]
	For $\pi \in S_{n-1}$ and $\epsilon \in E_{n-1}$, define $\sigma_{\pi,\epsilon} \in \mathrm{O}(n)$ by $\sigma_{\pi,\epsilon}e_i = \epsilon_i e_{\pi(i)}$ for $i < n$ and $\sigma_{\pi,\epsilon}e_n = e_n$.
	Then $W = \{ \sigma_{\pi,\epsilon} \mid \pi \in S_{n-1},$ $\epsilon \in E_{n-1} \}$.
	We use the convention that $\pi(n) \defeq n$ for $\pi \in S_{n-1}$ and $\epsilon_n \defeq 1$ for $\epsilon \in E_{n-1}$.

	For $\pi \in S_{n-1}$ and $\epsilon \in E_{n-1}$ let
	\[
		D(\pi, \epsilon) \defeq \{ j \in \{ 1, \ldots, n-1 \} \mid \epsilon_{j-1} \pi(j-1) > \epsilon_j \pi(j) \}
	\]
	with the convention that $\epsilon_0 \pi(0) \defeq -\epsilon_2 \pi(2)$.

	\begin{proposition}[{\cite[Proposition 8.2.2]{BjornerBrenti}}]\label{prop:typeDdes}
	For all $\sigma_{\pi,\epsilon}\in W$ we have
	\[
		\Dr(\sigma_{\pi,\epsilon}) = D(\pi,\epsilon) \, .
	\]
	\end{proposition}

	For a proposition $P$, we use the symbol
	\[
		[P] \defeq \begin{cases}
			1 & \text{if $P$ is true,} \\
			0 & \text{if $P$ is false.}
		\end{cases}
	\]

	\begin{proposition}\label{prop:typeD}
		Fix integers $a_1, \ldots, a_{n-1}$ such that $\lvert a_1 \rvert \leq a_2 \leq \cdots \leq a_{n-1} \neq 0$.
		Let
		\[ \begin{split}
			C \defeq \{ x \in \mathbf{R}^n \mid {}
			& \forall \, \pi \in S_{n-1},\ \epsilon \in E_{n-1} : \\
			& \epsilon_1 a_1 x_{\pi(1)} + \cdots + \epsilon_{n-1} a_{n-1} x_{\pi(n-1)} \leq x_n \} \, .
		\end{split} \]
		Let
		\[
			\Gamma \defeq
			\{
				x \in \mathbf{Z}^n \mid x_1 \equiv \cdots \equiv x_{n-1} \mod (2)
			\} .
		\]
		The generating function of $C$ with respect to $\Gamma$ is
		\[
			f_C = \frac{1}{1-z_n} \sum_{\pi \in S_{n-1}} \sum_{\epsilon \in E_{n-1}} \frac{
				\prod_{j \in D(\pi,\epsilon)}
					\left( z_{\pi(1)}^{-\epsilon_1} z_n^{-a_1} \right)^{[j=2]}
					\left( \prod_{i=j}^{n-1} z_{\pi(i)}^{\epsilon_i} z_n^{a_i} \right)^{1 + [j \geq 3]}
			}{
				\prod_{j=1}^{n-1} \left(
					1 - \left( z_{\pi(1)}^{-\epsilon_1} z_n^{-a_1} \right)^{[j=2]}
					\left( \prod_{i=j}^{n-1} z_{\pi(i)}^{\epsilon_i} z_n^{a_i} \right)^{1 + [j \geq 3]}
				\right)
			} \, ,
		\]
		where $z_1, \ldots, z_n$ are the coordinates corresponding to the standard lattice $\mathbf{Z}^n \subset \mathbf{R}^n$.
	\end{proposition}

	Note that the conditions on the $a_i$ are normalizations rather than restrictions.

	\begin{proof}
		The cone $C$ is symmetric and monoconditional for $W$.
		Let
		\[
			F \defeq \{ x \in \mathbf{R}^n \mid \lvert x_1 \rvert \leq x_2 \leq \cdots \leq x_{n-1} \} \, ,
		\]
		a fundamental domain for $W$.
		Our simple generators $s_1, \ldots, s_{n-1}$ defined at the beginning of this section are the simple generators of $W$ corresponding to the facets of $F$.
		Let $x_0 \defeq (a_1, \ldots, a_n, -1) \in F$.
		By the proof of \autoref{8UQHag7u},
		\begin{align*}
			C_+
			&= \{ x \in F \mid (x, x_0) \leq 0 \}
			= \{ x \in F \mid a_1 x_1 + \cdots + a_{n-1} x_{n-1} \leq x_n \} \\
			&= \{ x \in \mathbf{R}^n \mid Ax \geq 0 \} \, ,
		\end{align*}
		where
		\[
			A \defeq \begin{pmatrix} \begin{tikzpicture}[x=(-90:1.5em),y=(0:2.5em),dash pattern=on 0pt off 1ex,line cap=round,very thick]
				\draw (2,2) node[fill=white][fill=white]{$1$} -- (5,5) node[fill=white]{$1$};
				\draw (2,1) node[fill=white]{$-1$} -- (5,4) node[fill=white]{$-1$};
				\draw (1,1) node[fill=white]{$1$};
				\draw (1,2) node[fill=white]{$1$};
				\draw (1,3) node[fill=white]{$0$} -- (1,6) node[fill=white]{$0$};
				\draw (2,3) node[fill=white]{$0$} -- (2,6) node[fill=white]{$0$} -- (5,6) node[fill=white]{$0$} -- cycle;
				\draw (3,1) node[fill=white]{$0$} -- (5,1) node[fill=white]{$0$} -- (5,3) node[fill=white]{$0$} -- cycle;
				\draw (6,1) node[fill=white]{$-a_1$} -- (6,5) node[fill=white]{$-a_{n-1}$};
				\draw (6,6) node[fill=white]{$1$};
			\end{tikzpicture} \end{pmatrix} .
		\]
		The inverse of $A$ is
		\[
			A^{-1} = \begin{pmatrix} \begin{tikzpicture}[x=(-90:1.5em),y=(0:2.5em),dash pattern=on 0pt off 1ex,line cap=round,very thick]
				\draw (1,1) node[fill=white]{$1/2$};
				\draw (1,2) node[fill=white]{$-1/2$};
				\draw (2,1) node[fill=white]{$1/2$} -- (5,1) node[fill=white]{$1/2$};
				\draw (2,2) node[fill=white]{$1/2$} -- (5,2) node[fill=white]{$1/2$};
				\draw (3,3) node[fill=white]{$1$} -- (5,3) node[fill=white]{$1$} -- (5,5) node[fill=white]{$1$} -- cycle;
				\draw (1,3) node[fill=white]{$0$} -- (1,6) node[fill=white]{$0$};
				\draw (2,3) node[fill=white]{$0$} -- (2,6) node[fill=white]{$0$} -- (5,6) node[fill=white]{$0$} -- cycle;
				\draw (6,1) node[fill=white]{$\Sigma_1/2$};
				\draw (6,2) node[fill=white]{$\Sigma_2'/2$};
				\draw (6,3) node[fill=white]{$\Sigma_3$} -- (6,5) node[fill=white]{$\Sigma_{n-1}$};
				\draw (6,6) node[fill=white]{$1$};
			\end{tikzpicture} \end{pmatrix} ,
		\]
		where $\Sigma_j \defeq a_j + \cdots + a_{n-1}$ and $\Sigma_2' \defeq \Sigma_2 - a_1$.
		Hence the $\Gamma$-primitive generators of $C_+$ are the column vectors $b_1, \ldots, b_n$ of the matrix
		\[
			B \defeq \begin{pmatrix} \begin{tikzpicture}[x=(-90:1.5em),y=(0:2.5em),dash pattern=on 0pt off 1ex,line cap=round,very thick]
				\draw (1,1) node[fill=white]{$1$};
				\draw (1,2) node[fill=white]{$-1$};
				\draw (2,1) node[fill=white]{$1$} -- (5,1) node[fill=white]{$1$};
				\draw (2,2) node[fill=white]{$1$} -- (5,2) node[fill=white]{$1$};
				\draw (3,3) node[fill=white]{$2$} -- (5,3) node[fill=white]{$2$} -- (5,5) node[fill=white]{$2$} -- cycle;
				\draw (1,3) node[fill=white]{$0$} -- (1,6) node[fill=white]{$0$};
				\draw (2,3) node[fill=white]{$0$} -- (2,6) node[fill=white]{$0$} -- (5,6) node[fill=white]{$0$} -- cycle;
				\draw (6,1) node[fill=white]{$\Sigma_1$};
				\draw (6,2) node[fill=white]{$\Sigma_2'$};
				\draw (6,3) node[fill=white]{$2\Sigma_3$} -- (6,5) node[fill=white]{$2\Sigma_{n-1}$};
				\draw (6,6) node[fill=white]{$1$};
			\end{tikzpicture} \end{pmatrix} .
		\]
		As $\det(B) = 2^{n-2} = \lvert \mathbf{Z}^n / \Gamma \rvert$ it follows that $C_+$ is unimodular.
		Note that $b_1, \ldots, b_n$ are enumerated in the unique way such that $b_j \not\in H_j$ for $j < n$, where $H_j$ is the reflection hyperplane for the reflection $s_j$.
		Hence by \autoref{jx6Xq2Cz} the generating function of $C$ is
		\[
			f_C = \sum_{\sigma \in W} \frac{\prod_{j \in \Dr(\sigma)} z^{\sigma b_j}}{(1-z^{\sigma b_1}) \cdots (1-z^{\sigma b_n})} \, .
		\]
		Let
		\[
			b_{ij} \defeq \begin{cases}
				-1 & \text{if } i = 1,\ j = 2, \\
				1 & \text{if } j \leq i < n \text{ or } i = j = n, \\
				0 & \text{if } i < j \geq 2, \\
				\Sigma_1 & \text{if } i = n,\ j = 1, \\
				\Sigma_2' & \text{if } i = n,\ j = 2, \\
				2\Sigma_j & \text{if } i = n,\ 2 < j < n
			\end{cases}
		\]
		be the $i$th component of $b_j$, i.e., the $(i,j)$th component of $B$.	
		Thus
		\begin{align*}
			f_C
			&= \sum_{\sigma \in W} \frac{\prod_{j \in \Dr(\sigma)} z^{\sigma b_j}}{\prod_{j=1}^n (1-z^{\sigma b_j})} \\
			&= \sum_{\pi \in S_{n-1}} \sum_{\epsilon \in E_{n-1}} \frac{\prod_{j \in \Dr(\sigma_{\pi,\epsilon})} z^{\sigma_{\pi,\epsilon} b_j}}{\prod_{j=1}^n (1-z^{\sigma_{\pi,\epsilon} b_j})} \\
			&= \sum_{\pi \in S_{n-1}} \sum_{\epsilon \in E_{n-1}} \frac{\prod_{j \in D(\pi, \epsilon)} \prod_{i=1}^n z_{\pi(i)}^{\epsilon_i b_{ij}}}{\prod_{j=1}^n \left(1 - \prod_{i=1}^n z_{\pi(i)}^{\epsilon_i b_{ij}} \right)} \\
			&= \frac{1}{1-z_n} \sum_{\pi \in S_{n-1}} \sum_{\epsilon \in E_{n-1}} \frac{
				\prod_{j \in D(\pi,\epsilon)}
					\left( z_{\pi(1)}^{-\epsilon_1} z_n^{-a_1} \right)^{[j=2]}
					\left( \prod_{i=j}^{n-1} z_{\pi(i)}^{\epsilon_i} z_n^{a_i} \right)^{1 + [j \geq 3]}
			}{
				\prod_{j=1}^{n-1} \left(
					1 - \left( z_{\pi(1)}^{-\epsilon_1} z_n^{-a_1} \right)^{[j=2]}
					\left( \prod_{i=j}^{n-1} z_{\pi(i)}^{\epsilon_i} z_n^{a_i} \right)^{1 + [j \geq 3]}
				\right)
			} \, .
			\qedhere
		\end{align*}
	\end{proof}

\bibliographystyle{amsplain}
\bibliography{symmetriccones}

\providecommand{\doi}[1]{\href{http://dx.doi.org/#1}{\nolinkurl{doi:#1}}}
  \providecommand{\arxiv}[1]{\href{http://arxiv.org/abs/#1}{\nolinkurl{arXiv:#%
1}}}
\providecommand{\bysame}{\leavevmode\hbox to3em{\hrulefill}\thinspace}
\providecommand{\MR}{\relax\ifhmode\unskip\space\fi MR }
\providecommand{\MRhref}[2]{%
  \href{http://www.ams.org/mathscinet-getitem?mr=#1}{#2}
}
\providecommand{\href}[2]{#2}
\begin{thebibliography}{10}

\bibitem{PAVII}
George~E. Andrews, Peter Paule, and Axel Riese, \emph{Mac{M}ahon's partition
  analysis {VII}: constrained compositions}, $q$-Series with applications to
  combinatorics, number theory, and physics (Bruce~C. Berndt and Ken Ono,
  eds.), American Mathematical Society, 2001, pp.~11--27.

\bibitem{BeckGesselLeeSavage}
Matthias Beck, Ira~M. Gessel, Sunjoung Lee, and Carla~D. Savage,
  \emph{Symmetrically constrained compositions}, Ramanujan J. \textbf{23}
  (2010), 355--369, \doi{10.1007/s11139-010-9232-7}.

\bibitem{BeckRobins}
Matthias Beck and Sinai Robins, \emph{Computing the continuous discretely:
  integer-point enumeration in polyhedra}, Springer, 2007,
  \doi{10.1007/978-0-387-46112-0}.

\bibitem{BjornerBrenti}
Anders Bjorner and Francesco Brenti, \emph{Combinatorics of {C}oxeter groups},
  Springer, 2005, \doi{10.1007/3-540-27596-7}.

\bibitem{Bourbaki}
N.~Bourbaki, \emph{Groupes et alg{\`e}bres de {L}ie: chapitres 4, 5 et 6},
  Masson, 1981, \doi{10.1007/978-3-540-34491-9}.

\bibitem{BME1}
Mireille Bousquet-M{\'e}lou and Kimmo Eriksson, \emph{Lecture hall partitions},
  Ramanujan J. \textbf{1} (1997), no.~1, 101--111,
  \doi{10.1023/A:1009771306380}.

\bibitem{brentieulerian}
Francesco Brenti, \emph{{$q$}-{E}ulerian polynomials arising from {C}oxeter
  groups}, European J. Combin. \textbf{15} (1994), no.~5, 417--441,
  \doi{10.1006/eujc.1994.1046}.

\bibitem{ChowGessel}
Chak-On Chow and Ira~M. Gessel, \emph{On the descent numbers and major indices
  for the hyperoctahedral group}, Adv. Appl. Math. \textbf{38} (2007),
  275--310, \doi{10.1016/j.aam.2006.07.003}.

\bibitem{Humphreys}
James~E. Humphreys, \emph{Reflection groups and {C}oxeter groups}, Cambridge
  University Press, 1990.

\bibitem{PS2}
Thomas~W. Pensyl and Carla~D. Savage, \emph{Lecture hall partitions and the
  wreath products ${C}_k \wr {S}_n$}, submitted (2012), available at
  \url{http://www4.ncsu.edu/~savage/}.

\bibitem{steingrimsson}
Einar Steingr{\'{\i}}msson, \emph{Permutation statistics of indexed
  permutations}, European J. Combin. \textbf{15} (1994), no.~2, 187--205,
  \doi{10.1006/eujc.1994.1021}.

\end{thebibliography}

\end{document}